\documentclass[11pt,reqno]{amsart}

\usepackage{amsmath}
\usepackage{amssymb}
\usepackage{amscd}
\usepackage{latexsym}
\usepackage{graphicx}
\usepackage{color}

\newcommand{\wh}{\widehat}

\newcommand{\ispa}[1]{\langle \,#1 \,\rangle }

\newcommand{\ol}{\overline}

\newcommand{\mb}{\mathbb}

\newcommand{\fcal}{\mathcal{F}}

\newcommand{\hcal}{\mathcal{H}}

\newcommand{\rcal}{\mathcal{R}}

\newcommand{\ucal}{\mathcal{U}}

\newcommand{\re}{{\rm Re}\,}

\newcommand{\spec}{{\rm spec\hspace{0.1mm}}}
\newcommand{\dsp}{\displaystyle}

\newcommand{\Hom}{{\rm Hom}}

\newcommand{\iu}{\sqrt{-1}}

\newtheorem{thm}{{\sc Theorem}}[section]
\newtheorem{cor}[thm]{{\sc Corollary}}
\newtheorem{lem}[thm]{{\sc Lemma}}
\newtheorem{prop}[thm]{{\sc Proposition}}
\newtheorem{defin}[thm]{{\sc Definition}}

\begin{document}

\title[Eigenvalues, absolute continuity and localizations]
{Eigenvalues, absolute continuity and localizations \\
for periodic unitary transition operators}
\author{Tatsuya Tate}
\address{Mathematical Institute, Graduate School of Sciences, Tohoku University, 
Aoba, Sendai 980-8578, Japan. }
\email{tate@m.tohoku.ac.jp}
\thanks{The author is partially supported by JSPS Grant-in-Aid for Scientific Research (No. 25400068, No. 24340031).}
\date{\today}

\renewcommand{\thefootnote}{\fnsymbol{footnote}}
\renewcommand{\theequation}{\thesection.\arabic{equation}}
\renewcommand{\labelenumi}{{\rm (\arabic{enumi})}}
\renewcommand{\labelenumii}{{\rm (\alph{enumii})}}
\numberwithin{equation}{section}

\begin{abstract}
The localization phenomenon for periodic unitary transition operators on a Hilbert space consisting of square summable functions 
on an integer lattice with values in a finite dimensional Hilbert space, 
which is a generalization of the discrete-time quantum walks with 
constant coin matrices, are discussed. 
It is proved that a periodic unitary transition operator has an eigenvalue if and only if the corresponding 
unitary matrix-valued function on a torus has an eigenvalue which does not depend on the points on the torus. 
It is also proved that the continuous spectrum of a periodic unitary transition operators is absolutely continuous. 
As a result, it is shown that the localization happens if and only if there exists an eigenvalue, and 
when there exists only one eigenvalue, the long time limit of transition probabilities 
coincides with the point-wise norm of the projection of the initial state to the eigenspace. 
The results can be applied to certain unitary operators on a Hilbert space 
on a covering graph, called a topological crystal, over a finite graph. 
An analytic perturbation theory for matrices in several complex variables is 
employed to show the result about absolute continuity for periodic unitary transition operators. 
\end{abstract}

\maketitle

\section{Introduction}\label{INTRO}

The term {\it discrete-time quantum walks} is recently used to mean probability distributions, 
which is called transition probabilities, defined by using unitary operators. In this context, 
the unitary operators are often called unitary transition operators or unitary time-evolutions. 
Originally the specific class of unitary operators in typical models are formulated 
in quantum physics (\cite{ADZ}) and computer science (\cite{AAKV},\cite{ABNVW}, \cite{NV}). 
As in the theory of usual random walks, clarifying the asymptotic behavior,
as time goes to infinity, of the transition probability is one of main issues for the investigation of quantum walks. 
It is well-known that asymptotic behavior of quantum walks in typical models 
are quite different from that of usual random walks. The weak-limit distributions of quantum walks are 
usually ballistic (\cite{Ko1}, \cite{WKKK}), and the pointwise asymptotics is also quite different 
from that of the classical random walks (\cite{ST}). 
One of peculiar properties of quantum walks is a {\it localization}, which is a phenomenon 
that the transition probabilities, at some points, do not tend to zero as time goes to infinity. 
Interesting is that the localization phenomenon happens for quantum walks with very simple classes of 
unitary transition operators, even on the one-dimensional integer lattice, as discussed in \cite{IKK}, \cite{IK}, \cite{IKS}. 
There are numerous literatures on the localization phenomenon for quantum walks and it 
would be hard to mention all of them, and thus we just refer \cite{KOS} 
where the localization properties of a quantum walk on certain infinite graphs called spidernets, 
not on the usual integer lattices, is discussed. 

To be more concrete, 
let $Z$ be a countable set and let $W$ be a finite dimensional Hilbert space 
with the inner product $\ispa{\cdot,\cdot}_{W}$. 
Let $U$ be a unitary operator on the Hilbert space $\ell^{2}(Z,W)$ 
consisting of all square summable functions on $Z$ with values in $W$ with the inner product defined by 
\begin{equation}\label{inn1}
\ispa{f,g}=\sum_{x \in Z}\ispa{f(x),g(x)}_{W} \quad (f,g \in \ell^{2}(Z,W)). 
\end{equation}
For a non-negative integer $n$, a non-zero vector $u \in \ell^{2}(Z,W)$ and a point $x \in Z$, 
we define the quantity $p_{n}(u;x)$ by 
\begin{equation}\label{T1}
p_{n}(u;x)=\|(U^{n}u)(x)\|_{W}^{2}, 
\end{equation}
where $\|\cdot \|_{W}$ denotes the norm on $W$ defined by the fixed Hermitian inner product. 
It is straightforward to see that the sum of $p_{n}(u;x)$ over all $x \in Z$ equals $\|u\|^{2}$, 
the norm square of $u$ in $\ell^{2}(Z,W)$, 
and hence $p_{n}(u;x)$ defines a probability distribution on $V$ when $\|u\|=1$. 
For this reason, we call, in this note, 
$p_{n}(u;x)$ the transition probability for the unitary transition operator $U$ with the initial state $u$ 
even when $u$ is not a unit vector. 

\begin{defin}\label{LOC}
The transition probability of the unitary transition operator $U$ 
with initial state $u$ is said to be localized at 
a point $x \in Z$ if $\dsp \limsup_{n \to \infty}p_{n}(u;x) >0$. 
\end{defin}

In this paper, we mainly consider the integer lattice $\mb{Z}^{d}$ of rank $d \geq 1$ 
for the set $Z$ and the {\it periodic unitary transition operators} on $H=\ell^{2}(\mb{Z}^{d},W)$ 
for the unitary operator $U$, which is defined as follows. 
For each $y \in \mb{Z}^{d}$ and $\phi \in W$, 
we define $\delta_{y} \otimes \phi \in H$ by the following formula. 
\[
(\delta_{y} \otimes \phi)(x)=
\begin{cases}
\phi & (\mbox{when $x=y$}), \\
0 & (\mbox{otherwise}). 
\end{cases}
\]
\begin{defin}\label{UTOD}
A unitary operator $U$ on $H=\ell^{2}(\mb{Z}^{d},W)$ is said to be a periodic unitary transition operator 
if the following two conditions are satisfied. 
\begin{itemize}
\item The unitary operator $U$ commutes with the natural action of the abelian group $\mb{Z}^{d}$ on $H$. 
\item There exists a finite set $S \subset \mb{Z}^{d}$, called the set of steps, 
such that for any $x \in \mb{Z}^{d}$, $y \in \mb{Z}^{d} \setminus (x +S)$ and any $\phi \in W$, we have 
$U(\delta_{x}\otimes \phi)(y)=0$.
\end{itemize}
\end{defin}
It is obvious that unitary transition operators in the usual model of discrete-time quantum walks 
with constant coin matrices and products of finite number of such operators are periodic unitary transition operators. 
To analyze periodic unitary transition operators, It is natural to use the Fourier transform 
because of their invariance under the action of $\mb{Z}^{d}$. 
Let $T^{d}$ be the $d$-dimensional torus in $\mb{C}^{d}$ defined by 
\[
T^{d}=\{z=(z_{1},\ldots,z_{d}) \in \mb{C}^{d}\,;\,|z_{j}|=1\ (j=1,\ldots,d)\}. 
\]
Let $\hcal=L^{2}(T^{d},W)$ be the Hilbert space consisting of all square integrable 
functions on $T^{d}$ (with respect to the Lebesgue measure on $T^{d}$)  with values in $W$. 
The normalized Lebesgue measure on $T^{d}$ is denoted by $\nu_{d}$. 
The inner product on $\hcal$ is defined in a usual manner with the Hermitian 
inner product on $W$. Let $\fcal:\hcal \to H$ be the Fourier transform defined by 
\begin{equation}\label{FT1}
\fcal(f)(x)=\int_{T^{d}} z^{x}  f(z)\,d\nu_{d}(z) \quad (f \in \hcal), 
\end{equation}
where we write $z^{x}=z_{1}^{x_{1}} \cdots z_{d}^{x_{d}}$ for a point 
$z=(z_{1},\ldots,z_{d})$ in the complex torus $T_{\mb{C}}^{d}=(\mb{C} \setminus \{0\})^{d}$ 
and a lattice point $x=(x_{1},\ldots,x_{d})$ in $\mb{Z}^{d}$. 
The Fourier transform $\fcal$ is a unitary operator with the inverse $\fcal^{*}=\fcal^{-1}:H \to \hcal$ given by 
\begin{equation}
(\fcal^{*}g)(z)=\sum_{x \in \mb{Z}^{d}} g(x)z^{-x} \quad (z \in T^{d}). 
\end{equation}
For a periodic unitary transition operator $U$, we define a unitary operator $\ucal$ on $\hcal$ by $\ucal=\fcal^{*} U \fcal$. 
Then the unitary operator $\ucal$ can be written in the form
\begin{equation}\label{UT1}
(\ucal f)(z)=\wh{U}(z)f(z) \ \ (f \in \hcal,\ z \in T^{d}) 
\end{equation}
with a unitary operator $\wh{U}(z)$ on the finite dimensional vector space $W$ given by 
\begin{equation}
\wh{U}(z)\phi=\ucal(1 \otimes \phi)(z)\quad (\phi \in W), 
\end{equation}
where, for a scalar function $f \in L^{2}(T^{d})$ and a vector $\phi \in W$, we define $f \otimes \phi \in \hcal$ by 
\[
(f \otimes \phi)(z)=f(z)\phi \quad (z \in T^{d}). 
\]
Our main theorem is then described as follows. 
\begin{thm}\label{main0}
For a periodic unitary transition operator $U$ on $\ell^{2}(\mb{Z}^{d},W)$, we have the following. 
\begin{enumerate}
\item $U$ has an eigenvalue $\omega$ if and only if the matrix-valued function $\wh{U}(z)$ on $T^{d}$ has 
an eigenvalue $\omega$ for all points $z$ in $T^{d}$. Hence the number of eigenvalues of $U$ is finite. 
\item The continuous spectrum of $U$ is absolutely continuous with respect to 
the Lebesgue measure on the unit circle $S^{1}$ in $\mb{C}$. 
\end{enumerate}
\end{thm}

\noindent Therefore, by the absolute continuity of the continuous spectrum, we have 

\begin{cor}\label{cor1}
The periodic unitary transition operator $U$ has no eigenvalues 
if and only if we have 
$\dsp \limsup_{n \to \infty}p_{n}(u;x)=0$ for any $u \in \ell^{2}(\mb{Z}^{d},W)$ and $x \in \mb{Z}^{d}$. 
Therefore, the localization phenomenon does not happen to the quantum walks 
defined by such unitary operators. 
\end{cor}

\noindent A simple observation shows the following. 

\begin{cor}\label{cor2}
Suppose that the periodic unitary transition operator $U$ has 
only one eigenvalue $\omega$, and let $\pi$ be the orthogonal projection onto the eigenspace 
corresponding to the eigenvalue $\omega$. 
Then, for any $u \in H$ and $x \in \mb{Z}^{d}$, the limit $\displaystyle \lim_{n \to \infty}p_{n}(u;x)$ exists and 
the following formula holds. 
\begin{equation}\label{TAF1}
\lim_{n \to \infty}p_{n}(u;x)=\|(\pi u)(x)\|_{W}^{2}.
\end{equation}
\end{cor}

It would be necessary here to give some remarks on these results. 
First, it was pointed out to the author by professor Alain Joye that 
the singular continuous spectrum of (general unitary operator) $U$ does not affect 
the {\it long-time average} of the transition probabilities, 
\begin{equation}\label{TA1}
\ol{p}(u;x)=\lim_{N \to \infty}\frac{1}{N}\sum_{n=1}^{N} p_{n}(u;x). 
\end{equation}
Indeed, let $U$ be a general unitary operator on $\ell^{2}(Z, W)$ (for a general countable set $Z$) and 
let $\pi_{\lambda}$ be the orthogonal projection onto the eigenspace corresponding to 
an eigenvalue $\lambda$ of $U$. 
Then, Wiener's formula (\cite{Kat}) tells us that the limit $\ol{p}(u;x)$ exists and satisfies
\begin{equation}\label{Wiener1}
\ol{p}(u;x)=\sum_{\mbox{{\tiny $\lambda$: eigenvalue of $U$}}}\|(\pi_{\lambda}u)(x)\|_{W}^{2}. 
\end{equation}
From this it is obvious that if $U$ has an eigenvalue then a localization takes place at some 
point $x \in Z$ and an initial state $u \in \ell^{2}(Z, W)$. 
In Corollary \ref{cor2}, we take the long-time limit (but not the average) itself of the transition probability 
under the assumption that $U$ has only one eigenvalue. For example, it is well-known that 
the three-state Grover walk on $\mb{Z}$ and the square of the two-dimensional Grover walk satisfy this assumption. 

The absolute continuity of continuous spectrum for certain class of self-adjoint operators 
has been obtained in the work of G\'{e}rard-Nier \cite{GN}. 
The theory developed in \cite{GN} is very powerful, and 
it might be possible to apply this theory to the generator of the periodic unitary transition operators. 
However, taking a logarithm of the periodic unitary transition operators would cause singularities 
of the function corresponding to $\wh{U}$ introduced above and hence it does not seem so straightforward to apply it. 
Since our situation is rather easier, 
it would be reasonable to give a direct and constructive proof of absolute continuity. 

Since in the setting up described above the vector space $W$ can be arbitrarily chosen, 
Theorem $\ref{main0}$ is applicable to certain unitary operators on topological crystal over finite graphs. 
As is discussed in Section $\ref{LGU}$, some unitary transition operators are defined over 
the set of oriented edges. But, the set of vertices, as in the classical case, might be used to 
formulate quantum system over a topological crystal. Therefore, we use a setting described as follows. 
As before, let $Z$ be a countable set and let $W_{0}$ be a finite dimensional Hilbert space with the 
inner product $\ispa{\cdot,\cdot}_{W_{0}}$. Suppose that a free abelian group $\Gamma$ 
of finite rank $d$ acts on $Z$ freely and the quotient set $Z_{0}=Z/\Gamma$ is finite. 
Let $p:Z \to Z_{0}$ be the natural projection. 
Identifying the finite set $Z_{0}$ with a fixed fundamental set $F_{0}$ of the action of $\Gamma$, 
we may regard $Z_{0}$ as a subset of $Z$. 
Let $C(Z,W_{0})$ be the vector space of all function from $Z$ to $W_{0}$. 
Let $U$ be a unitary operator on the Hilbert space $\ell^{2}(Z,W_{0})$ satisfying the following two conditions. 
\begin{enumerate}
\item[(i)] $U$ commutes with the action of $\Gamma$ on $\ell^{2}(Z,W_{0})$. 
\item[(ii)] There exists a finite set $S \subset \Gamma$ such that, 
for any $x_{0}, y_{0} \in Z_{0}$, $\alpha \in \Gamma$, $\beta \in \Gamma\setminus \alpha S$ and $\phi \in W_{0}$, 
we have $U(\delta_{\alpha x_{0}} \otimes \phi)(\beta y_{0})=0$. 
\end{enumerate}
Here we note that the group operation on $\Gamma$ is written as a multiplication because 
it would be better to describe the action of $\Gamma$ on $Z$. 
The condition (ii) means that if the quantum walker with the time evolution $U$ start at 
a point in the fundamental set, $\alpha F_{0}$, the walker can arrive in one step at points in  
$\bigcup_{\gamma \in S}\gamma \alpha F_{0}$. Thus, the finite set $S$ depends on the choice of the fundamental set $F_{0}$. 
Let $\wh{\Gamma}$ be the group of $U(1)$-character on $\Gamma$. 
For each $\chi \in \wh{\Gamma}$, we set
\[
C_{\chi}(Z,W_{0})=\{f:Z \to W_{0} \,;\, f(\alpha x)=\chi(\alpha) f(x) \ (x \in Z, \alpha \in \Gamma)\}. 
\]
The space $C_{\chi}(Z,W_{0})$ is finite dimensional and is isomorphic to $\ell^{2}(Z_{0},W_{0})$, 
and it is equipped with a natural Hermitian inner product inherited from $\ell^{2}(Z_{0},W_{0})$. 
By the assumption (i), (ii), the operator $U$ is continuous with respect to the topology of point-wise 
convergence, and thus it defines a linear map on the space of all $W_{0}$-valued function on $Z$. 
For each $\chi \in \wh{\Gamma}$, we define a unitary operator $U_{\chi}$ on $C_{\chi}(Z,W_{0})$ 
as the restriction of $U$ to $C_{\chi}(Z,W_{0})$.

\begin{cor}\label{cor3}
Let $U$ be a unitary operator on $\ell^{2}(Z,W_{0})$ satisfying the above two conditions 
and, for $\chi \in \wh{\Gamma}$, let $U_{\chi}$ be the unitary operator on $C_{\chi}(Z,W_{0})$ defined as above. 
Then we have the following. 
\begin{enumerate}
\item $U$ has an eigenvalue $\omega$ if and only if the unitary operator $U_{\chi}$ on $C_{\chi}(Z,W_{0})$ has 
an eigenvalue $\omega$ for all points $\chi \in \wh{\Gamma}$. Hence the number of eigenvalues of $U$ is finite. 
\item The continuous spectrum of $U$ is absolutely continuous with respect to 
the Lebesgue measure on the unit circle $S^{1}$ in $\mb{C}$. 
\end{enumerate}
Hence the statements in Corollaries $\ref{cor1}$, $\ref{cor2}$ holds true for $U$. 
\end{cor}

Organization of the present paper is as follows. First, in Section $\ref{LGU}$, 
the corollaries are proved. Main reason for applicability of Theorem $\ref{main0}$ 
to the setting-up in Corollary $\ref{cor3}$ is that, under the assumption (i), (ii), 
the operator $U$ becomes a periodic unitary transition operator on $\ell^{2}(Z,W_{0}) \cong \ell^{2}(\Gamma,W)$ 
with $W=\ell^{2}(Z_{0},W_{0})$. An example, the Grover walks on topological crystals will be 
explained in Section $\ref{LGU}$. 
In Section $\ref{UTO}$, some simple but important properties of 
periodic unitary transition operators are summarized. The item (1) in Theorem $\ref{main0}$ is an easy 
consequence of the formula $\eqref{UT1}$, but its proof is written, for completeness of the presentation, 
in Section $\ref{EIG}$. In Section $\ref{ABS}$, absolute continuity of the continuous spectrum is discussed. 
In this section we employ the analytic perturbation 
theory developed by Baumg\"{a}rtel (\cite{Ba}) and use a coarea formula to 
construct a density function of the spectral measures.

\vspace{10pt}

\noindent{\bf Acknowledgments.}\hspace{3pt}
The author would like to thank professor Alain Joye for his comments about
the Wiener formula, professor Yusuke Higuchi for 
the comments on the work \cite{HN} where a lemma essentially the same as our Lemma \ref{analN} 
below is used in a similar purpose, and professor Serge Richard for stimulate discussion.

\section{Proofs of corollaries}\label{LGU}
Let $Z$ be, as in Section $\ref{INTRO}$, a countable set and let $W$ be a complex vector space of dimension $D$ 
with a fixed Hermitian inner product $\ispa{\cdot, \cdot}_{W}$. 
Let $U$ be a unitary operator on the Hilbert space $H=\ell^{2}(Z,W)$ with the inner product given by $\eqref{inn1}$. 
Let $\spec(U)$ and $\spec(U)_{p}$ denote the spectrum and the set of eigenvalues of $U$, respectively. 
Since $U$ is unitary, $\spec(U)$ is contained in the unit circle $S^{1}$ in the complex plane, 
and the spectral resolution of $U$ takes the form 
\begin{equation}
U=\int_{S^{1}}\lambda\,dE(\lambda), 
\end{equation}
where $E$ is a projection-valued measure on $S^{1}$. 
If $\omega \in \spec(U)_{p}$, then $\pi_{\omega}=E(\{\omega\})$ is the projection onto 
the eigenspace corresponding to $\omega$. 
We say that $u \in H$ has absolutely continuous spectral measure if the measure $\|E(\cdot)u\|^{2}$ 
is absolutely continuous with respect to the Lebesgue measure on the circle. 
The following lemma can be proved by using the Riemann-Lebesgue lemma. 
\begin{lem}\label{CS1}
Let $u \in H$. Suppose that $u$ has absolutely continuous spectral measure. 
Then, for any $x \in Z$, we have $\dsp \lim_{n \to \infty}p_{n}(u;x)=0$. 
\end{lem}
\noindent Using Lemma $\ref{CS1}$ it is easy to show the following. 
\begin{lem}\label{EG1}
Suppose that a unitary operator $U$ on $H$ has only one eigenvalue 
and that the continuous spectrum of $U$ is absolutely continuous. 
Let $\pi$ be the orthogonal projection onto the eigenspace of $U$. 
Then we have $\displaystyle \lim_{n \to \infty}p_{n}(u;x)=\|(\pi u)(x)\|_{W}^{2}$ for any 
$u \in H$ and $x \in Z$. 
\end{lem}
\begin{proof}
The transition probability $p_{n}(u;x)$ can be written as 
\begin{equation}\label{decT1}
p_{n}(u;x)=p_{n}(\pi u;x)+p_{n}(\pi_{c}u;x)+r_{n}(u;x)=\|(\pi u) (x)\|_{W}^{2}+p_{n}(\pi_{c}u;x)+r_{n}(u;x), 
\end{equation}
where we set $\pi_{c}=I-\pi$ and the quantity $r_{n}(u;x)$ is defined by 
\begin{equation}\label{decT2}
r_{n}(u;x)=2\re\left(
\ispa{(U^{n}\pi u)(x),(U^{n}\pi_{c}u)(x)}_{W}
\right).
\end{equation}
By the assumption on the absolute continuity and Lemma $\ref{CS1}$, 
we see $\displaystyle \lim_{n \to \infty}p_{n}(\pi_{c}u;x)=0$. 
Cauchy-Schwarz inequality shows $r_{n}(w;x) \leq 2\|\pi w (x)\|_{\mb{C}^{D}} p_{n}(\pi_{c}w;x)^{1/2}$ which 
tends to zero as $n \to \infty$.
\end{proof}

Corollaries $\ref{cor1}$ and $\ref{cor2}$ are direct consequences 
of the Wiener formula $\eqref{Wiener1}$, Lemmas $\ref{CS1}$, $\ref{EG1}$ and Theorem $\ref{main0}$. 

Next, we give a proof of Corollary $\ref{cor3}$. To begin with, we prepare notation. 
Let $Z$ be, as before, a countable set and let $\Gamma$ be a free abelian group of finite rank $d$. 
We suppose that $\Gamma$ acts on $Z$ freely with the finite quotient $Z_{0}=Z/\Gamma$. 
Let $F_{0} \subset Z$ be a complete set of representatives of elements in the orbit space $Z_{0}$, 
which is often called a fundamental set of the action of $\Gamma$ on $Z$. 
The natural projection $p:Z \to Z_{0}$ maps $F_{0}$ bijectively onto $Z_{0}$. 
We set $q=(p|_{F_{0}})^{-1}:Z_{0} \to F_{0}$. 
Let $U$ be a unitary operator on $\ell^{2}(Z,W_{0})$ satisfying the conditions (i), (ii) described before Corollary $\ref{cor3}$. 
Define the operator $r:\ell^{2}(Z,W_{0}) \to \ell^{2}(\Gamma,\ell^{2}(Z_{0},W_{0}))$ by 
\[
r(f)(\alpha)(x_{0})=f(\alpha q(x_{0})) \quad (f \in \ell^{2}(Z,W_{0}),\, \alpha \in \Gamma,\,x_{0} \in Z_{0}). 
\]
The operator $r$ is unitary with the natural Hermitian inner product on $\ell^{2}(\Gamma,\ell^{2}(Z_{0},W_{0}))$.  
Since the action of $\Gamma$ on $\ell^{2}(Z,W_{0})$ is intertwined by the unitary operator 
$r:\ell^{2}(Z,W_{0}) \to \ell^{2}(\Gamma,\ell^{2}(Z_{0},W_{0}))$ with the natural action of $\Gamma$ on 
$\ell^{2}(\Gamma, \ell^{2}(Z_{0},W_{0}))$, the unitary operator $rUr^{*}$ on $\ell^{2}(\Gamma,\ell^{2}(Z_{0},W_{0}))$ 
is still $\Gamma$-equivariant. 
\begin{lem}\label{finPro}
The unitary operator $rUr^{*}$ on $\ell^{2}(\Gamma,\ell^{2}(Z_{0},W_{0}))$ is a 
periodic unitary transition operator in the sense of Definition $\ref{UTOD}$. 
\end{lem}
\begin{proof}
Condition (i) for $U$ ensures that $rUr^{*}$ commutes with the action of $\Gamma$ on $\ell^{2}(\Gamma,\ell^{2}(Z_{0},W_{0}))$. 
We take a finite set $S \subset \Gamma$ as in the condition (ii) imposed on $U$. 
Let $\alpha \in \Gamma$ and $\beta \in \Gamma \setminus \alpha S$, 
and let $f_{0} \in \ell^{2}(Z_{0},W_{0})$ be arbitrary. 
We have to compute $rUr^{*}(\delta_{\alpha} \otimes f_{0})(\beta)$. 
The function $\delta_{\alpha} \otimes f_{0}$ is in $\ell^{2}(\Gamma,\ell^{2}(Z_{0},W_{0}))$, 
and the function in $\ell^{2}(Z,W_{0})$ corresponding to it is given by 
\begin{equation}\label{sstar}
r^{*}(\delta_{\alpha} \otimes f_{0})=\sum_{x_{0} \in Z_{0}}\delta_{\alpha q(x_{0})} \otimes f_{0}(x_{0}). 
\end{equation}
Thus, we get 
\[
Ur^{*}(\delta_{\alpha} \otimes f_{0})=\sum_{x_{0} \in Z_{0}} U(\delta_{\alpha q(x_{0})} \otimes f_{0}(x_{0})).
\]
The condition (ii) imposed on $U$ implies that, for $\beta \not\in \alpha S$ and $y_{0} \in Z_{0}$, 
the value at $\beta q(y_{0})$ of each of the summand in the right-hand side vanishes. 
Hence $rUr^{*}(\delta_{\alpha} \otimes f_{0})(\beta)=0$ for each $\beta \not\in \alpha S$, which 
shows that $rUr^{*}$ is a periodic unitary transition operator on $\ell^{2}(\Gamma,\ell^{2}(Z_{0},W_{0}))$.
\end{proof}
The map $r$ is naturally extended to the map $r:C(Z,W_{0}) \to C(\Gamma,\ell^{2}(Z_{0},W_{0}))$. 
Since $r$ commutes with the action of $\Gamma$, we have the following isomorphism
\begin{equation}\label{twistS}
r:C_{\chi}(Z,W_{0}) \stackrel{\cong}{\to} \chi \otimes \ell^{2}(Z_{0},W_{0}), 
\end{equation}
where, for a function $\mu$ on $\Gamma$ and $f_{0} \in \ell^{2}(Z_{0},W_{0})$, 
the function $\mu \otimes f_{0} \in C(\Gamma, \ell^{2}(Z_{0},W_{0}))$ is defined by 
\[
(\mu \otimes f_{0})(\alpha) (x_{0})=\mu(\alpha) f_{0}(x_{0}) \quad (\alpha \in \Gamma,\, x_{0} \in Z_{0}). 
\]
The map $r$ on $C_{\chi}(Z,W_{0})$ composed with the map
\[
\chi \otimes \ell^{2}(Z_{0},W_{0}) \ni \chi \otimes f_{0} \mapsto f_{0}=(\chi \otimes f_{0})(1) \in \ell^{2}(Z_{0},W_{0})
\]
gives a unitary operator
\begin{equation}\label{twistC}
r_{\chi}:C_{\chi}(Z,W_{0}) \stackrel{\cong}{\to} \ell^{2}(Z_{0},W_{0}),  
\quad r_{\chi}(f)=q^{*}f. 
\end{equation}
The Fourier transform $\fcal:L^{2}(\wh{\Gamma},\ell^{2}(Z_{0},W_{0})) \to \ell^{2}(\Gamma, \ell^{2}(Z_{0},W_{0}))$ 
and its inverse is given by 
\[
(\fcal f)(\alpha)(x_{0})=\int_{\wh{\Gamma}} \rho(\alpha) f(\rho)(x_{0}) \,d\nu(\rho),\quad 
(\fcal^{-1}g)(\chi)(x_{0})=\sum_{\gamma \in \Gamma} \chi(\gamma^{-1}) g(\gamma)(x_{0}), 
\]
where $f \in L^{2}(\wh{\Gamma},\ell^{2}(Z_{0},W_{0}))$, $g \in \ell^{2}(\Gamma, \ell^{2}(Z_{0},W_{0}))$, 
$\chi \in \wh{\Gamma}$, $\alpha \in \Gamma$ and $x_{0} \in Z_{0}$. 
\begin{lem}\label{twist}
For each $\chi \in \wh{\Gamma}$, 
we have $\wh{rUr^{*}}(\chi)=r_{\chi} U_{\chi} r_{\chi}^{*}$, where $U_{\chi}$ is the restriction of $U$ 
on $C_{\chi}(Z,W_{0})$. 
\end{lem}
\begin{proof}
We take $f_{0} \in \ell^{2}(Z_{0},W_{0})$. Then we have 
$\fcal(\pmb{1} \otimes f_{0})=\delta_{1} \otimes f_{0}$, where $\pmb{1}$ is the 
function taking the value identically $1$ on $\wh{\Gamma}$ and $\delta_{1}$ is the delta function 
at the identity element $1 \in \Gamma$. 
We see, for $\chi \in \wh{\Gamma}$,  
\[
\begin{split}
\wh{rUr^{*}}(\chi)f_{0} =\ucal (1 \otimes f_{0}) (\chi)= \sum_{\gamma \in \Gamma}
\chi(\gamma) rUr^{*}(1 \otimes f_{0}) (\gamma^{-1})=
\sum_{\gamma \in \Gamma}
\chi(\gamma) rUr^{*}(\delta_{\gamma} \otimes f_{0}) (1). 
\end{split}
\]
Now, the sum $\sum_{\gamma \in \Gamma} \chi(\gamma) \delta_{\gamma} \otimes f_{0}$ 
converges in the topology of point-wise convergence on $C(\Gamma, \ell^{2}(Z_{0},W_{0}))$ 
to the function $\chi \otimes f_{0}$. Since, by the condition (ii), $rUr^{*}$ is continuous in this topology, 
we have 
\[
\sum_{\gamma \in \Gamma}
\chi(\gamma) rUr^{*}(\delta_{\gamma} \otimes f_{0})=rUr^{*}(\chi \otimes f_{0})=r_{\chi}U_{\chi}r_{\chi}^{*} f_{0}, 
\]
which completes the proof. 
\end{proof}
Now, Corollary $\ref{cor3}$ follows from Lemmas $\ref{finPro}$, $\ref{twist}$ and Theorem $\ref{main0}$. 

We remark that, in the above proof, the space $Z$ is regarded as a product $\Gamma \times Z_{0}$. 
This method is, for the setting-up in graph theory, not quite natural because it does not take the graph structure 
into account. For the setting in graph theory, 
it would be more transparent to use a unitary operator constructed as follows. We fix a function $s_{\chi}:Z \to U(1)$ 
satisfying $s_{\chi}(\alpha x)=\chi(\alpha) s_{\chi}(x)$ for all $\alpha \in \Gamma$, $x \in Z$. 
Then the operator $S_{\chi}:\ell^{2}(Z_{0},W_{0}) \to C_{\chi}(Z,W_{0})$ given by 
\[
(S_{\chi}f_{0})(x)=s_{\chi}(x)f(p(x)) \quad (f_{0} \in \ell^{2}(Z_{0},W_{0}),\, x \in Z)
\]
is a unitary operator. In our setting above, we formulated problems on $\ell^{2}(\Gamma, \ell^{2}(Z_{0},W_{0}))$ 
rather than on $\ell^{2}(Z,W_{0})$. Thus it would be reasonable and useful to use the 
operator $r_{\chi}:C_{\chi}(Z,W_{0}) \to \ell^{2}(Z_{0},W_{0})$. The operator $S_{\chi}^{-1}U_{\chi}S_{\chi}$ is 
unitarily equivalent to $r_{\chi}U_{\chi}r_{\chi}^{*}$ by the operator $r_{\chi}S_{\chi}$ which 
is just the multiplication by the function $q^{*}s_{\chi}$. 
Therefore, the spectral structure of $U_{\chi}$ or $\wh{rUr^{*}}(\chi)$ 
is the same as that of $S_{\chi}^{-1}U_{\chi}S_{\chi}$. 

Let us describe an example where Corollary $\ref{cor3}$ can be applied, 
Grover walks on topological crystals. 
In the following, 
the terminology, definitions and the descriptions basically follow that of \cite{KS} and \cite{S}. 
For a graph $Y$, we denote the set of vertices and the set of oriented edges by $V(Y)$, $E(Y)$, respectively. 
For any $e \in E(Y)$, the origin $o(e)$ and the terminus $t(e)$ are the vertices which are connected by the 
oriented edge $e$ in an obvious direction. For any vertex $y \in V(Y)$, we set $E(Y)_{y}=\{e \in E(Y) \,;\, o(e)=x\}$. 
Then the degree of $y \in V(Y)$, denoted by $\deg (y)$, is given by the number of elements in $E(Y)_{y}$. 
An orientation of the graph $Y=(V(Y),E(Y))$ is a subset $E^{o}(Y)$ of $E(Y)$ 
such that $E(Y)$ is written as the disjoint union $E(Y)=E^{o}(Y) \cup \ol{E^{o}(Y)}$.  

Let $X$ be a topological crystal (\cite{S}) over a connected finite graph $X_{0}$ 
with an abstract periodic lattice $\Gamma$, namely, $X$ is an infinite locally finite 
connected graph on which a free abelian group $\Gamma$ of rank $d$ ($0<d<+\infty$) acts 
freely and without inversion such that the quotient graph is $X_{0}$. 
The natural projection from $X$ onto $X_{0}$ is denoted by $p:X \to X_{0}$. 
A linear map $U$ on the space $C(E(X))$ consisting of all complex-valued function on $E(X)$ defined by 
\begin{equation}\label{grover1}
(U\phi)(e)=-\phi(\ol{e}) +\frac{2}{\deg (t(e))} \sum_{e' \in E(X)_{t(e)}} \phi(e') \qquad (\phi \in C(E(X)),\, e \in E(X)), 
\end{equation}
is often called a {\it Grover walk} (\cite{W}). 
See also \cite{HKSS} where a generalization, called a {\it twisted Szegedy walk} is formulated. 
The restriction of $U$ on $\ell^{2}(E(X))$, which is also denoted by $U$, defines a unitary operator on $\ell^{2}(E(X))$, 
and $U$ commutes with the action of $\Gamma$. Since the value $(U\phi)(e)$ depends only on $\phi(e')$ and $\phi(\ol{e})$ 
with the edges $e'$ adjacent to $e$, the operator $U$ satisfies 
the conditions (i), (ii) described before the statement of Corollary $\ref{cor3}$. 

Let us fix $\chi \in \wh{\Gamma}$ and compute the operator $\wh{U}(\chi) \cong U_{\chi}$ on $\ell^{2}(E(X_{0}))$. 
To compute it, we need to choose a fundamental set $F_{0}$ in the set of oriented edges $E(X)$ of the action of $\Gamma$. 
We fix a spanning tree $T_{0}$ in $X_{0}$ and its lift $T$ in $X$. 
For any $u \in E(X_{0})$, there exists a unique edge $q(u) \in E(X)$ with the properties $p(q(u))=u$ and $o(q(u)) \in V(T)$. 
Then the finite set $F_{0}=q(E(X_{0}))$ in $E(X)$ is a fundamental set of 
the action of $\Gamma$ on $E(X)$. Note that in general $q(\ol{u}) \neq \ol{q(u)}$. But if $u \in E(T_{0})$ then we have $q(\ol{u})=\ol{q(u)}$. 
We follow the construction of a unitary operator $S_{\chi}:\ell^{2}(E(X_{0})) \to C_{\chi}(X)$ given in \cite{KS}. 
We fix $\omega \in \Hom(\Gamma,\mb{R})$ such that $\chi=\exp (2\pi \iu \omega)$. 
Let $\mu:H_{1}(X_{0},\mb{Z}) \to \Gamma$ be the surjective homomorphism characterized by the 
identity 
\[
t(\tilde{c})=\mu(\ispa{c}) o(\tilde{c}),
\]
where $c$ is a closed path in $X_{0}$, $\tilde{c}$ is a lift of $c$, 
$\ispa{c}$ is the homology class of $c$ (with coefficients in $\mb{Z}$) and $t(\tilde{c})$, $o(\tilde{c})$ denote 
the terminus and origin of the path $\tilde{c}$ in $X$. 
We extend $\mu$ as a linear map $\mu_{\mb{R}}:H_{1}(X_{0},\mb{R}) \to \Gamma \otimes \mb{R}$. 
Since the dual space of $H_{1}(X_{0},\mb{R})$ is naturally identified with the cohomology group $H^{1}(X_{0},\mb{R})$, 
the transpose 
$\mu_{\mb{R}}^{*}:\Hom(\Gamma,\mb{R}) \cong (\Gamma \otimes \mb{R})^{*} \to H^{1}(X_{0},\mb{R})$ 
is an injective linear map. 
We take a one-form $\eta$, 
which is a function on $E(X_{0})$ such that $\eta(\ol{e})=-\eta(e)$, 
whose cohomology class is $\mu_{\mb{R}}^{*}\omega \in H^{1}(X_{0},\mb{R})$. 
We have $\omega (\mu(\ispa{c}))=\eta(c)$ for any closed path $c$, where, for $c=(e_{1},\ldots,e_{n})$, 
we write $\eta(c)=\sum_{i=1}^{n}\eta(e_{i})$. We remark that $\eta(c)$ does not depend on the choice of $\eta$ 
representing the cohomology class $\mu_{\mb{R}}^{*}\omega$. 
We fix $e_{*} \in E(T)$ and define a function $s_{\eta}:E(X) \to \mb{C}$ by 
\[
s_{\eta}(e)=e^{
2\pi \iu (p^{*}\eta)(\gamma(e_{*},e))
}, 
\]
where $\gamma(e_{*},e)$ denotes any path from $o(e_{*})$ to $t(e)$ in $X$ such that 
the first edge of $\gamma(e_{*},e)$ is $e_{*}$ and the last edge is $e$. 
Since $\mu(\ispa{p(\tilde{c})})=1$ for any closed path $\tilde{c}$ on $X$, 
the right-hand side in the definition of $s_{\eta}$ does not depend on 
the choice of such a path from $o(e_{*})$ to $t(e)$, 
and we have $s_{\eta} \in C_{\chi}(E(X))$. 
We note that $s_{\eta}$ depends on the choice of $1$-form $\eta$ 
representing $\mu_{\mb{R}}^{*}\omega$ and even the choice of $\omega$. 
Now the linear map 
$S_{\chi}:\ell^{2}(E(X_{0})) \to C_{\chi}(E(X))$ defined 
by 
\[
(S_{\chi} f_{0})(e)=s_{\eta}(e) f_{0}(p(e)) \quad (f_{0} \in \ell^{2}(E(X_{0})),\, e \in E(X))
\]
is a unitary isomorphism. 
It is straightforward to see
\[
(S_{\chi}^{-1}U_{\chi}S_{\chi} f_{0})(e_{0})=-(e^{2\pi \iu \eta}f_{0})(\ol{e_{0}}) 
+\frac{2}{\deg t(e_{0})} \sum_{u \in E(X_{0})_{t(e_{0})}} 
(e^{2\pi \iu \eta} f_{0})(u) \quad (e_{0} \in E(X_{0})),  
\]
which is (unitarily equivalent to) the twisted Szegedy walk whose spectral structure is discussed in \cite{HKSS}.

\section{Simple properties of periodic unitary transition operators}\label{UTO}

For $i=1,\ldots,d$, let $\tau_{i}$ be the shift operator on $H=\ell^{2}(\mb{Z}^{d},\mb{C}^{D})$ defined by 
\[
(\tau_{i}g)(x)=g(x-e_{i}) \quad (g \in H,\ x \in \mb{Z}^{d}), 
\]
where $\{e_{1},\ldots,e_{d}\}$ denotes the standard basis of $\mb{Z}^{d}$ over $\mb{Z}$. 
Then, $\tau_{i}$'s are unitary operators on $H$. 
We note that $\tau_{i}$ and $\tau_{j}$ commutes for any $i,j$, and 
hence, for $\dsp \alpha=\sum_{i=1}^{d}\alpha_{i}e_{i}$, we may write 
$\tau^{\alpha}=\tau_{1}^{\alpha_{1}} \cdots \tau_{d}^{\alpha_{d}}$. 
The action of $\mb{Z}^{d}$ on $H$ is given by the operators $\tau^{\alpha}$.  
\begin{lem}\label{PUT1}
Let $U$ be a periodic unitary transition operator on $H$ with the set of steps $S \subset \mb{Z}^{d}$. 
For each $\alpha \in S$, define a linear map $C(\alpha)$ on $\mb{C}^{D}$ by 
\[
C(\alpha)\phi=U(\delta_{0} \otimes \phi)(\alpha) \quad (\phi \in \mb{C}^{D}). 
\]
Then, $U$ can be written in the following form. 
\[
U=\sum_{\alpha \in S}\tau^{\alpha}C(\alpha). 
\]
\end{lem}
\begin{proof}
It is straightforward to see that, for any $\phi \in \mb{C}^{D}$ and $x \in \mb{Z}^{d}$, 
\[
U(\delta_{x} \otimes \phi)=\sum_{\alpha \in S} \tau^{\alpha} \delta_{x} \otimes C(\alpha) \phi
=\left(
\sum_{\alpha \in S}\tau^{\alpha} C(\alpha)
\right) (\delta_{x} \otimes \phi). 
\]
Since the set $\{\delta_{x} \otimes \phi\}_{x \in \mb{Z}^{d},\phi \in \mb{C}^{D}}$ spans $H$, we conclude the assertion. 
\end{proof}
The operator $\fcal^{*}\tau_{i}\fcal$, where $\fcal:\hcal \to H$ is the Fourier transform defined in $\eqref{FT1}$ and $\fcal^{*}$ is its inverse, 
is the multiplication operator defined by the coordinate function $z_{i}$, where $z=(z_{1},\ldots,z_{d}) \in T^{d}$. 
Hence, the operator $\ucal=\fcal^{*}U\fcal$ is written in the form $\eqref{UT1}$ with the matrix $\wh{U}(z)$ given by the formula
\begin{equation}\label{MVF}
\wh{U}(z)\phi=\sum_{\alpha \in S}z^{\alpha}C(\alpha)\phi=\ucal(1 \otimes \phi)(z) \quad (\phi \in \mb{C}^{D}).
\end{equation}
Therefore, by Lemma $\ref{PUT1}$ and the assumption that $U$ is a unitary operator, we see 
\[
\sum_{\alpha,\beta \in S,\,\alpha-\beta=\gamma}C(\beta)^{*}C(\alpha)=\delta_{\gamma,0}I
\]
for any $\gamma \in \mb{Z}^{d}$. From this and $\eqref{MVF}$, it follows that the matrix $\wh{U}(z)$ is 
unitary for any $z \in T^{d}$. We note that $\wh{U}(z)$ is a Laurent polynomial in $z \in T^{d}$ and 
hence it is naturally extended to the complex torus $T_{\mb{C}}^{d}=(\mb{C} \setminus \{0\})^{d}$. 
However, in general, $\wh{U}(z)$ is not normal when $z$ is not in the real torus $T^{d}$. 
We remark that, for a bounded operator $A$ on $H=\ell^{2}(\mb{Z}^{d},\mb{C}^{D})$ commuting with 
the action of $\mb{Z}^{d}$, the bounded operator $\fcal^{*}A \fcal$ on $\hcal=L^{2}(T^{d},\mb{C}^{D})$ commutes with 
multiplication operators by continuous functions on $T^{d}$. 
The characteristic polynomial 
\begin{equation}\label{char1}
\chi(\zeta,z)=\det(\zeta-\wh{U}(z)) \quad (z \in T_{\mb{C}}^{d},\,\zeta \in \mb{C})
\end{equation}
of $\wh{U}(z)$ is a polynomial in $\zeta$ of degree $D$ with coefficients in the ring of Laurent polynomials in $z \in T_{\mb{C}}^{d}$. 
\begin{lem}\label{sp1}
We have the following. 
\[
\spec(U)=\{\lambda \in S^{1}\,;\,
\mbox{There exists a point $z \in T^{d}$ such that $\lambda$ is an eigenvalue of $\wh{U}(z)$}
\}.
\]
\end{lem}
\begin{proof}
It is easy to show that the spectrum $\spec(U)$ is contained in the set in the right-hand side above. 
Thus, we prove that $\spec(U)$ contains the set in the right-hand side. 
Suppose that $\lambda$ is in the resolvent set of $U$ and set $V=\{z \in T^{d}\,;\,\chi(\lambda,z) \neq 0\}$. 
Then, $\rcal=(\lambda -\ucal)^{-1}$ is a bounded operator on $\hcal$. 
Let $W \subset T^{d}$ be a Borel set such that $\nu_{d}(W)=1$ and $\wh{R}(z)\phi:=(\lambda -\ucal)^{-1}(1 \otimes \phi)(z)$ is defined 
for any $z \in W$ and any $\phi \in \mb{C}^{D}$. Then we have $\phi=(\lambda -\wh{U}(z))\wh{R}(z)\phi$ 
for any $z \in W$ and $\phi \in \mb{C}^{D}$. Thus, $W \subset V$. This shows that $\nu_{d}(V)=1$ and 
that $V$ is an open dense set such that $T^{d} \setminus V$ does not contain any non-empty open set. 
Therefore we write $\wh{R}(z)=(\lambda -\wh{U}(z))^{-1}$ for $z \in V$. 
We claim that $V=T^{d}$. 
To show this, suppose contrary that $T^{d} \setminus V \neq \emptyset$ and we take $z_{o} \in T^{d} \setminus V$. 
We take $\phi_{o} \in \mb{C}^{D}$ such that $\|\phi_{o}\|_{\mb{C}^{D}}=1$ and $(\lambda -\wh{U}(z_{o}))\phi_{o}=0$. 
For any positive integer $n$, there exists a neighborhood $U_{n}$ of $z_{o}$ in $T^{d}$  such that 
$\|(\lambda -\wh{U}(z))\phi_{o}\|_{\mb{C}^{D}}<1/(n+1)$ for any $z \in U_{n}$. 
Then, for any $z \in U_{n} \cap V$, we have
\[
1=\|\wh{R}(z)(\lambda -\wh{U}(z))\phi_{o}\|_{\mb{C}^{D}} \leq \|\wh{R}(z)\|_{{\rm op}} \|(\lambda -\wh{U}(z))\phi_{o}\|_{\mb{C}^{D}}
\leq \|\wh{R}(z)\|_{{\rm op}} /(n+1), 
\]
where $\|\cdot \|_{{\rm op}}$ denotes the operator norm. 
For each $n$, we fix $z_{n} \in U_{n} \cap V$. 
One can find $\phi_{n} \in \mb{C}^{D}$ with $\|\phi_{n}\|_{\mb{C}^{D}}=1$ 
such that $n<\|\wh{R}(z_{n})\phi_{n}\|_{\mb{C}^{D}}$. 
There exists a neighborhood $V_{n}$ of $z_{n}$ such that $V_{n} \subset U_{n} \cap V$ and 
$n<\|\wh{R}(z)\phi_{n}\|_{\mb{C}^{D}}$ for any $z \in V_{n}$. We take a continuous function $f_{n}$ 
on $T^{d}$ such that the support of $f_{n}$ is contained in $V_{n}$ and 
$\|f_{n}\|_{L^{2}(T^{d})}=1$. We set $v_{n}=f_{n} \otimes \phi_{n} =f_{n}(1\otimes \phi_{n})\in \hcal$ which satisfies $\|v_{n}\|=1$. 
We see $(\rcal v_{n})(z)=f_{n}(z) \wh{R}(z)\phi_{n}$ for any $z \in W$. Then, 
\[
\|\rcal\|_{{\rm op}}^{2} \geq \|\rcal v_{n}\|^{2}=\int_{V_{n}} |f_{n}(z)|^{2} \|\wh{R}(z)\phi_{n}\|_{\mb{C}^{D}}^{2} d\nu_{d}(z) \geq n^{2}, 
\]
which contradicts the fact that $\rcal$ is a bounded operator. This shows that $V=T^{d}$, and hence $\lambda$ is not an 
eigenvalue of $\wh{U}(z)$ for all $z \in T^{d}$. 
\end{proof}

\section{Eigenvalues of periodic unitary transition operators}\label{EIG}

In this section and the next section, we frequently use the following lemma. 
This lemma is essentially the same as Lemma 4.4 in \cite{HN}. However, we give 
its proof for convenience. 

\begin{lem}\label{analN}
Let $\Delta \subset \mb{C}^{d}$ be a domain. 
Let $f$ be a holomorphic function on $\Delta$. 
Suppose that $f^{-1}(0) \cap T^{d}$ has positive Lebesgue measure on $T^{d}$. 
Then $f$ is identically zero on $\Delta$. 
\end{lem}
\begin{proof}
First, let us assume that $\Delta$ is a polydisk in $\mb{C}^{d}$, that is, $\Delta$ is a product of $d$ disks in $\mb{C}$. 
We use an induction on the dimension $d$. When $d=1$, the assertion is trivial since $f^{-1}(0) \cap S^{1}$ is finite if $f$ is not identically zero. 
Suppose that, for $d \geq 2$, the assertion is true in the dimension $d-1$. 
Let $\Delta \subset \mb{C}^{d}$ be a polydisk and let $f$ be a holomorphic function on $\Delta$. 
Suppose that $f$ is not identically zero. We would like to show that $\nu_{d}(f^{-1}(0) \cap T^{d})=0$. 
We may assume that $\Delta \cap T^{d} \neq \emptyset$. 
In the following, a point in $\mb{C}^{d}$ is written as $(z,w)$ with $z \in \mb{C}^{d-1}$ and $w \in \mb{C}$. 
We set $\Delta=V_{d-1} \times V_{1}$ where $V_{d-1} \subset \mb{C}^{d-1}$ is a polydisk and $V_{1} \subset \mb{C}$ is a disk. 
For any $(z,w) \in \Delta$, we set $g_{w}(z)=h_{z}(w)=f(z,w)$. 
The functions $g_{w}$ and $h_{z}$ are holomorphic on $V_{d-1}$ and $V_{1}$, respectively. 
Fubini's theorem shows 
\[
\nu_{d}(f^{-1}(0) \cap T^{d})=\int_{S^{1}} \nu_{d-1}(g_{w}^{-1}(0) \cap T^{d-1})\,d\nu_{1}(w). 
\]
We set $M=\{w \in S^{1}\,;\,\nu_{d-1}(g_{w}^{-1}(0) \cap T^{d-1})>0\}$. 
We show that $\nu_{1}(M)=0$. 
By induction hypothesis, $w \in M$ if and only if $g_{w}$ is identically zero on $V_{d-1}$. 
Suppose that $M \neq \emptyset$ and we take $w_{o} \in M$. 
Since $g_{w_{o}}=0$ on $V_{d-1}$, $h_{z}(w_{o})=0$ for any $z \in V_{d-1}$. 
Since $f$ is not identically zero, one can take a point $z_{o} \in V_{d-1}$ such 
that $h_{z_{o}}$ is not identically zero on $V_{1}$. 
For such a point, one can take $z_{o}$ in $V_{d-1} \cap T^{d-1}$ by the induction hypothesis and 
the assumption that $f$ is not identically zero. 
Since $h_{z_{o}}$ is holomorphic on the disk $V_{1}$ in $\mb{C}$, $h_{z_{o}}^{-1}(0)$ is a discrete set containing $w_{o}$. 
There exists an open neighborhood $I$ of $w_{o}$ in $V_{1}$ such that $h_{z_{o}}^{-1}(0) \cap I=\{w_{o}\}$. 
We clearly have $M \cap I=\{w_{o}\}$. 
This means that $M$ is a discrete set in $S^{1} \cap V_{1}$ and hence $M$ is at most countable 
and its Lebesgue measure is zero. 
Now, let $\Delta \subset \mb{C}^{d}$ be a general domain and let $f$ be a holomorphic function on $\Delta$.  
Since $\Delta$ is second countable, one can find a countable open overing $\{\Delta_{n}\}_{n=1}^{\infty}$ of $\Delta$ 
such that each $\Delta_{n}$ is a polydisk. Suppose that $\nu_{d}(f^{-1}(0) \cap T^{d})>0$. 
Then, we have 
\[
0<\nu_{d}(f^{-1}(0) \cap T^{d}) \leq \sum_{n} \nu_{d}(\Delta_{n} \cap f^{-1}(0) \cap T^{d}), 
\]
and hence there exists an integer $n$ such that $\nu_{d}(\Delta_{n} \cap f^{-1}(0) \cap T^{d})>0$. 
Thus, from what we have just proved above, $f$ is identically zero on $\Delta_{n}$. Now the identity theorem for holomorphic 
functions shows that $f$ is identically zero on $\Delta$. 
\end{proof}

\begin{lem}\label{suff} 
Suppose that $U$ has an eigenvalue $\lambda$. 
Then $\lambda$ is an eigenvalue of $\wh{U}(z)$ for any $z \in T_{\mb{C}}^{d}$. 
\end{lem}
\begin{proof}
Let $0 \neq w \in \hcal$ be an eigenvector of $\ucal$ with the eigenvalue $\lambda$. 
Then, there exists a Borel set $W \subset T^{d}$ such that $\nu_{d}(W)=1$ and $\wh{U}(z)w(z)=\lambda w(z)$ for all $z \in W$. 
We may set $w(z)=0$ for $z \in T^{d} \setminus W$. 
We have $\chi_{\lambda}(z)=0$ for any $z \in T^{d} \setminus w^{-1}(0)$, 
where $\chi_{\lambda}(z):=\chi(\lambda,z)$ is a Laurent polynomial in $z \in T_{\mb{C}}^{d}$. 
Since $w \neq 0$, we have $\nu_{d}(T^{d} \setminus w^{-1}(0))>0$. 
Therefore, Lemma $\ref{analN}$ shows that the function $\chi_{\lambda}$ 
is identically zero on $T_{\mb{C}}^{d}$. Thus, $\lambda$ is an eigenvalue of $\wh{U}(z)$ 
for any $z \in T_{\mb{C}}^{d}$. 
\end{proof}

We have to show the converse to Lemma $\ref{suff}$. 
Suppose that $\lambda \in S^{1}$ is an eigenvalue of the unitary matrix $\wh{U}(z)$ for any $z \in T^{d}$.  
By Lemma $\ref{analN}$, we have $\chi(\lambda,z)=0$ for any $z \in T_{\mb{C}}^{d}$. 
Therefore, for each fixed $z \in T_{\mb{C}}^{d}$, the polynomial $\chi(\zeta,z)$ in $\zeta \in \mb{C}$ is 
divisible by $\zeta-\lambda$. 
However, the multiplicity of $\zeta-\lambda$ in $\chi(\zeta,z)$ depends on $z \in T_{\mb{C}}^{d}$. 
To handle this situation, 
let $R_{d}$ denote the ring of Laurent polynomials in $z=(z_{1},\ldots,z_{d})$ with complex coefficients. 
The quotient field of $R_{d}$, denoted by $k_{d}$, is the field of rational functions on $\mb{C}^{d}$. 
Let $R_{d}[\zeta]$, {\it resp.} $k_{d}[\zeta]$, be the ring of polynomials in one variable $\zeta$ 
with coefficients in $R_{d}$, {\it reps.} in $k_{d}$. 
We have $\chi \in R_{d}[\zeta] \subset k_{d}[\zeta]$. 
Since $\chi(\lambda,z)=0$ for any $z \in T_{\mb{C}}^{d}$, 
the polynomial $\zeta -\lambda \in k_{d}[\zeta]$ divides $\chi$ in $k_{d}[\zeta]$. 
\begin{lem}\label{fac1}
Let $p,q \in k_{d}[\zeta]$ be monic polynomials and let $r \in R_{d}[\zeta]$. If $r=pq$, then $p,q \in R_{d}[\zeta]$. 
\end{lem}
\begin{proof}
It follows from Lemma 2 in \cite[Supplement]{Ba} that each coefficient of $p$ and $q$ is holomorphic on $T_{\mb{C}}^{d}$. 
Thus, what we need is to prove that a rational function $m(z)$ in 
$z \in \mb{C}^{d}$ which is holomorphic on $T_{\mb{C}}^{d}$ 
is actually a Laurent polynomial in $z$. We use an induction on $d$. 
Let us consider the case where $d=1$. Let $g/f$ ($f,g \in \mb{C}[z]$) be a rational function holomorphic on $T_{\mb{C}}^{1}$. 
We may suppose that $f$ and $g$ have no common roots. Since $g/f$ is holomorphic on $T_{\mb{C}}^{1}$, it is possible for $f$ to 
have a root only at the origin. Thus, $f$ is a monomial and hence $g/f$ is a Laurent polynomial. 

Suppose that the assertion holds in the dimension $d-1$. 
For $z \in \mb{C}^{d}$, we write, as before, $z=(\xi,w)$ with $\xi \in \mb{C}$ and $w=(w_{1},\ldots,w_{d-1}) \in \mb{C}^{d-1}$. 
Let $m(\xi,w)$ be a rational function in $(\xi,w) \in \mb{C}^{d}$ holomorphic in $T_{\mb{C}}^{d}$. We write $m(\xi,w)=g(\xi,w)/f(\xi,w)$ with 
polynomials $f$ and $g$ having no common irreducible factors in $\mb{C}[\xi,w]$. 
The assertion is trivial if $f$ is constant. Thus, suppose that $f$ is non constant. 
First, suppose that $g$ is a non-zero constant. 
Since $m=g/f$ is holomorphic on $T_{\mb{C}}^{d}$, $f$ can not be zero on $T_{\mb{C}}^{d}$. Let $s=\xi w_{1}\cdots w_{d-1}$. 
By Hilbert's Nullstelensatz (see, for instance, \cite[Chapter 4]{CLO}), each irreducible factor of $f$ divides $s$, and hence $f$ is a monomial. 
Next, suppose that $g$ is not constant. 
We may suppose that $g$ has a positive degree in $\xi$, and we write $g=b_{0}\xi^{m}+\cdots+b_{m}$ 
with $m>0$, $b_{j} \in \mb{C}[w]$ and $b_{0} \neq 0$. If $f$ is constant in $\xi$, then $f \in \mb{C}[w]$, and 
$m=(b_{0}/f)\xi^{m}+\cdots +(b_{m}/f)$. 
Since $b_{j}/f$'s are rational functions on $\mb{C}^{d-1}$ holomorphic in $T_{\mb{C}}^{d-1}$, the inductive hypothesis 
shows that these are Laurent polynomials in $w$. Thus, we may suppose that $f$ and $g$ have positive degrees in $\xi$. 
Let $f$ be a polynomial in $\xi$ of degree $l>0$, and we write $f=a_{0}\xi^{l}+\cdots+ a_{l}$ with $a_{j} \in \mb{C}[w]$ and $a_{0} \neq 0$. 
Let $r \in \mb{C}[w]$ be the resultant of the polynomials $f$, $g$ with respect to $\xi$. 
(See, for instance, \cite[Chapter 3]{CLO} for the properties of resultants.) Since $f$ and $g$ have no common 
factors, $r$ is not the zero polynomial. 
Suppose that $a_{l} \neq 0$. 
We set $a=a_{l}r \in \mb{C}[w]$. 
Then the polynomial $a$ is not the zero polynomial and $U=\{w \in \mb{C}^{d-1}\,;\,a(w) \neq 0\}$ is an open dense subset in $\mb{C}^{d-1}$. 
Fix $w_{o} \in U \cap T_{\mb{C}}^{d-1}$. Since $r(w_{o}) \neq 0$, $f(\xi,w_{o})$ and $g(\xi,w_{o})$ have no common zeros as 
polynomials in $\xi \in \mb{C}$. 
Since $m=g/f$ is holomorphic on $T_{\mb{C}}^{d}$, $f(\xi,w_{o})=0$ only at $\xi=0$. 
However, since $a_{l}(w_{o}) \neq 0$, we have $f(\xi,w_{o}) \neq 0$ for any $\xi \in \mb{C}$, which is a contradiction to the assumption that 
$f$ has a positive degree in $\xi$. Thus, $a_{l}$ is zero. 
This shows that $f$ is divisible by $\xi$. We write $f=\xi^{k}h$ where $k>0$ and $h$ 
is a polynomial in $\mb{C}[\xi,w]$ which is not divisible by $\xi$. 
Then $\xi^{k}m=g/h$ is again a rational function holomorphic on $T_{\mb{C}}^{d}$. 
Applying the above discussion shows that $h$ must be a constant in $\xi$. 
Thus, $h \in \mb{C}[w]$. 
Therefore, the previous discussion shows that $\xi^{k}m=g(\xi,w)/h(w)$ is a Laurent polynomial in $(\xi,w)$. 
\end{proof}

\vspace{10pt}

\noindent{\it Proof of Theorem $\ref{main0}$, {\rm (1)}.}\hspace{2pt}
Suppose, as above, that $\lambda \in S^{1}$ is an eigenvalue of $\wh{U}(z)$ for any $z \in T^{d}$ and hence for any $z \in T_{\mb{C}}^{d}$. 
We divide $\chi$ by $\zeta -\lambda$ in $k_{d}[\zeta]$ and write $\chi(\zeta,z)=(\zeta -\lambda)^{m} q(\zeta,z)$, 
where $q$ is not divisible by $\zeta -\lambda$ in $k_{d}[\zeta]$. By Lemma $\ref{fac1}$, $q$ is in $R_{d}[\zeta]$. 
By the Hamilton-Cayley theorem, we have 
\begin{equation}\label{HC1}
0=\chi(\wh{U}(z),z)=(\wh{U}(z)-\lambda)^{m} q(\wh{U}(z),z)
\end{equation}
as a matrix for any $z \in T_{\mb{C}}^{d}$. We note that $q(\wh{U}(z),z)$ is not identically zero on $T^{d}$. 
Indeed, suppose contrary that $q(\wh{U}(z),z)=0$ for any $z \in T^{d}$. 
Then, for any fixed $z \in T^{d}$, $q(\zeta,z)$ is divisible by the minimal polynomial of the matrix $\wh{U}(z)$. 
Since $\wh{U}(z)$ has $\lambda$ as an eigenvalue, $q(\zeta,z)$ is divisible by $\zeta-\lambda$ 
for any $z \in T^{d}$. Hence $q(\lambda,z)=0$ for any $z \in T_{\mb{C}}^{d}$ by Lemma $\ref{analN}$. 
Thus, $q$ is divisible by $\zeta -\lambda$ in $k_{d}[\zeta]$ which is a contradiction. 
Therefore, $q(\wh{U}(z),z)$ is not identically zero. 
By the Fourier transform, the bounded operator $q(U,\tau)$ on $H$ is not identically zero. 
From $\eqref{HC1}$, it follows that $0=\chi(U,\tau)=(U-\lambda)^{m}q(U,\tau)$. 
We take $w \in H$ such that $h:=q(U,\tau)w$ is not zero. Since $(U-\lambda)^{m}h=0$, there exists a number $j$ ($1 \leq j <m$) such 
that $(U-\lambda)^{j}h \neq 0$ and $(U-\lambda)^{j+1}h=0$. Then $v=(U-\lambda)^{j}h$ is an eigenvector of $U$ with 
the eigenvalue $\lambda$. 
\hfill$\square$

To close this section, it would be worth mentioning that the proof given above also shows the following
\begin{prop}
Let $\lambda$ be an eigenvalue of the periodic unitary transition operator $U$ on $H=\ell^{2}(\mb{Z}^{d},W)$. 
We decompose the characteristic polynomial $\chi(\zeta,z)$ in $k_{d}[\zeta]$ 
in the form $\chi(\zeta,z)=(\zeta -\lambda)^{m}q(\zeta,z)$, where $q(\zeta,z) \in R_{d}[\zeta]$ is an 
irreducible polynomial in $R_{d}[\zeta]$. Then, for any $w \in \ell^{2}(\mb{Z}^{d},W)$, 
we have $q(U,\tau)w=0$ or 
$(U-\lambda)^{j}q(U,\tau)w$ is an eigenfunction of $U$ for some $j$ with $1 \leq j <m$. 
\end{prop}

\section{Absolute continuity of the continuous spectrum}\label{ABS}

Finally we discuss the absolute continuity of the continuous spectrum 
of a periodic unitary transition operator $U$. 
As before, let $\spec(U)_{p}=\{\omega_{1},\ldots,\omega_{K}\}$ be the set of eigenvalues of $U$ 
and let $\pi_{j}$ be the orthogonal projection onto the eigenspace of $U$ corresponding to $\omega_{j}$. 
We set $\pi_{p}=\pi_{1}+\cdots+\pi_{K}$ and $\pi_{c}=I-\pi_{p}$. If $\spec(U)_{p}=\emptyset$, then we set $\pi_{p}=0$. 
The characteristic polynomial $\chi$ defined in $\eqref{char1}$ 
is divisible by $\zeta -\omega_{j}$ in the ring of polynomials $k_{d}[\zeta]$ with coefficients in the field $k_{d}$ of 
rational functions in $z=(z_{1},\ldots,z_{d})$. 
The characteristic polynomial $\chi$ is monic in the sense that the coefficient of the leading term of $\chi$ is the identity in $k_{d}$. 
The ring $k_{d}[\zeta]$ is a unique factorization domain (UFD for short), and hence $\chi \in k_{d}[\zeta]$ can be decomposed into 
powers of monic irreducible elements in $k_{d}[\zeta]$. It should be remarked that the ring $R_{d}$ of Laurent polynomials in $z=(z_{1},\ldots,z_{d})$ is 
a UFD because it is a localization of the ring $\mb{C}[z_{1},\ldots,z_{d}]$ of polynomials in $z$. Thus, $R_{d}[\zeta]$ is also a UFD, 
and hence $\chi$ can be decomposed into irreducibles in $R_{d}[\zeta]$. However, 
we need to use some properties of the notion of discriminants which are valid for polynomials with coefficients in a field. 
Since the polynomials $\zeta -\omega_{j}$ are irreducible, 
the decomposition of $\chi$ into monic irreducible elements in $k_{d}[\zeta]$ can be written as 
\begin{equation}
\chi(\zeta,z)=\prod_{j=1}^{K}(\zeta -\omega_{j})^{m_{j}} \times \prod_{\rho=1}^{r}\pi_{\rho}(\zeta,z)^{n_{\rho}}, 
\end{equation}
where $\pi_{\rho} \in k_{d}[\zeta]$ are monic irreducible polynomials in $k_{d}[\zeta]$, 
and $\zeta-\omega_{j}$ ($j=1,\ldots,K$) and $\pi_{\rho}$ ($\rho=1,\ldots,r$) are mutually different. We set
\begin{equation}\label{decP2}
\pi(\zeta,z)=\prod_{j=1}^{K}(\zeta -\omega_{j}) \times \prod_{\rho=1}^{r}\pi_{\rho}(\zeta,z)
\end{equation}
so that $\chi^{-1}(0)=\pi^{-1}(0) \subset T_{\mb{C}}^{d+1}$. 
If $\pi_{\rho}$ ($1 \leq \rho \leq r$) is of the form $\zeta -\lambda$ with a constant $\lambda \in \mb{C}$, 
then $\chi(\lambda,z)=0$ for any $z \in T_{\mb{C}}^{d}$ and hence by Theorem $\ref{main0}$, {\rm (1)}, $\lambda$ 
is an eigenvalue of the operator $U$. Thus, $\lambda=\omega_{j}$ for some $j=1,\ldots,K$. 
However, $\pi_{\rho}$ is different from $\zeta-\omega_{j}$, and hence it is a contradiction. 
Thus, $\pi_{\rho}$ is not of the form $\zeta- \lambda$ with a constant $\lambda$. 
By Lemma $\ref{fac1}$, each polynomial $\pi_{\rho}$ in $\eqref{decP2}$ is in $R_{d}[\zeta]$. 
Let $d_{\pi}$ denote the discriminant of $\pi$, namely, $d_{\pi}=\pm\, \mbox{resultant of $\pi$ and $\partial_{\zeta}\pi$}$. 
$d_{\pi}$ is a Laurent polynomial in $z \in T_{\mb{C}}^{d}$. 
Note that $d_{\pi}$ is identically zero if and only if $\pi$ and $\partial_{\zeta}\pi$ has a common irreducible factor in $k_{d}[\zeta]$. 
Since $\pi$ is a product of mutually different irreducible elements in $k_{d}[\zeta]$, $d_{\pi}$ is not identically zero. 
Hence the discriminant set $D=d_{\pi}^{-1}(0) \subset T_{\mb{C}}^{d}$ 
is an algebraic variety such that $T_{\mb{C}}^{d} \setminus D$ is a connected open dense set (\cite[Chapter 1]{GR}). 
By Lemma $\ref{analN}$, $T^{d} \cap D$ is a closed set in $T^{d}$ having Lebesgue measure zero. 
For any $z \in T_{\mb{C}}^{d} \setminus D$, the equation $\pi(\zeta,z)=0$ has $K+L=\deg \pi$ distinct roots 
in $\mb{C}$ because $d_{\pi}=d_{\pi}(z)$ is the discriminant of $\pi(\zeta,z)$ in $\mb{C}[\zeta]$ when $z$ is fixed and 
hence $\pi(\zeta,z)$ does not have multiple roots for fixed $z \not\in D$. (See \cite{CLO} for the properties of discriminants.) 

\begin{lem}\label{eigen1}
Let $z_{o} \in T_{\mb{C}}^{d} \setminus D$. Then, there exists a neighborhood $V \subset T_{\mb{C}}^{d} \setminus D$ of $z_{o}$ and holomorphic 
functions $\lambda_{1}(z),\ldots,\lambda_{L}(z)$ on $V$ such that, for each $z \in V$, the roots of the equation $\pi(\zeta,z)=0$ 
is given by $\omega_{1},\ldots,\omega_{K}$, $\lambda_{1}(z),\ldots,\lambda_{L}(z)$. 
\end{lem}
\begin{proof}
Let $\omega_{1},\ldots,\omega_{K}$, $\lambda_{1},\ldots,\lambda_{L}$ be the roots of the equation $\pi(\zeta,z_{o})=0$, 
each of which is different from others. Let $\lambda_{o}$ denote one of $\lambda_{j}$'s. 
Since $z_{o} \not\in D$ and since $\pi(\lambda_{o},z_{o})=0$, we have $\partial_{\zeta}\pi (\lambda_{o},z_{o}) \neq 0$. 
Applying the implicit function theorem (\cite[Chapter 1]{GR}), there exists a neighborhood $V$ of $z_{o}$ and 
a holomorphic function $\lambda(z)$ on $V$ such that $\pi(\lambda(z),z)=0$ for any 
$z \in V$ and $\lambda(z_{o})=\lambda_{o}$. This shows the lemma. 
\end{proof}

The projections onto algebraic (generalized) eigenspaces (called the eigenprojection in \cite{Ka}) 
are also holomorphic near $z_{o} \in T^{d} \setminus D$. Namely, we have the following. 
\begin{lem}\label{eigenP}
Let $z_{o} \in T_{\mb{C}}^{d} \setminus D$. Let $V_{o}$ be an open neighborhood of $z_{o}$ in $T_{\mb{C}}^{d} \setminus D$ 
as in Lemma $\ref{eigen1}$. Then there is an open neighborhood $V_{1}$ of $z_{o}$ in $V_{o}$ and 
holomorphic projection-valued functions $R_{j}(z)$ $(1 \leq j \leq K)$, $P_{k}(z)$ $(1 \leq k \leq L)$ such 
that $R_{j}(z)$, $P_{k}(z)$ are the projections onto the algebraic eigenspaces corresponding to $\omega_{j}$, $\lambda_{k}(z)$, respectively, 
for each $z \in V_{1}$. 
\end{lem}
\begin{proof}
As before, let $\omega_{1},\ldots,\omega_{K}$, $\lambda_{1},\ldots,\lambda_{L}$ be mutually different eigenvalues of 
the matrix $\wh{U}(z_{o})$. Let $\lambda_{o}$ be one of them. 
We take a small circle $\Gamma$ around $\lambda_{o}$ such that $\Gamma$ and the disk bounded by $\Gamma$ do not 
contain any other eigenvalues. Let $\lambda_{o}(z)$ denote the eigenvalue of $\wh{U}(z)$ in $V_{o}$ such that $\lambda_{o}(z_{o})=\lambda_{o}$ 
which is holomorphic on $V_{o}$, 
and let $\lambda(z)$ be any other holomorphic function which is an eigenvalue of $\wh{U}(z)$ in $V_{o}$. 
Since $\lambda_{o}(z)$ and $\lambda(z)$ are continuous, there exists a neighborhood $V_{1}$ of $z_{o}$ in $V_{o}$ such that, 
for each $z \in V_{1}$, $\lambda(z)$ is not contained in $\Gamma$ and the disk bounded by $\Gamma$ and 
$\lambda_{o}(z)$ is inside $\Gamma$. 
Then, the eigenprojection $P(z)$ onto the algebraic eigenspace corresponding to $\lambda_{o}(z)$ 
along other algebraic eigenspaces is given by the contour integral (\cite[Chapter 1]{Ka})
\begin{equation}\label{cont}
P(z)=\frac{1}{2\pi i}\int_{\Gamma} (\zeta -\wh{U}(z))^{-1}\,d\zeta. 
\end{equation}
From this expression, $P(z)$ is holomorphic in $z \in V_{1}$. 
\end{proof}
We remark that, when $z_{o} \in T^{d} \setminus D$, $R_{j}(z)$ and $P_{k}(z)$ in Lemma $\ref{eigenP}$ are 
smooth functions on $T^{d} \cap V_{1}$ which are the orthogonal projections onto the geometric eigenspaces 
because $\wh{U}(z)$ is a unitary matrix for $z \in T^{d}$ . For $R_{j}(z)$, we have the following.

\begin{lem}\label{unitP1}
In the above, $R_{j}(z)$ $(1 \leq j \leq K)$ is extended to $T^{d} \setminus D$ as 
a smooth orthogonal projection-valued bounded function, and it is the orthogonal projection onto 
the eigenspace corresponding to the eigenvalue $\omega_{j}$ of $\wh{U}(z)$ for each $z \in T^{d} \setminus D$. 
\end{lem}
\begin{proof}
For each fixed $z \in T_{\mb{C}}^{d}$, 
let $R_{j}(z)$ be the eigenprojection onto the algebraic eigenspace corresponding to the eigenvalue $\omega_{j}$ of $\wh{U}(z)$ 
defined by a contour integral as in $\eqref{cont}$. 
$R_{j}(z)$ is well-defined as a single-valued function on $T_{\mb{C}}^{d}$. 
By Lemma $\ref{eigenP}$, it is holomorphic on $T_{\mb{C}}^{d} \setminus D$, and hence its restriction to $T^{d} \setminus D$ 
is smooth. For $z \in T^{d}$, $\wh{U}(z)$ is a unitary matrix and hence $R_{j}(z)$ is the orthogonal projection onto 
the (geometric) eigenspaces corresponding to $\omega_{j}$. 
Thus, the operator norm of $R_{j}(z)$ for $z \in T^{d}$ is bounded. 
\end{proof}
\begin{lem}\label{unitP2}
Let $R_{j}(z)$ be the orthogonal projection-valued functions on $T^{d} \setminus D$ in Lemma $\ref{unitP1}$. 
Then, the orthogonal projection $\pi_{j}$ onto the eigenspace of $\ucal$ 
corresponding to the eigenvalue $\omega_{j}$ is given by the following formula. 
\[
(\pi_{j}w)(z)=R_{j}(z)w(z) \quad (w \in L^{2}(T^{d},\mb{C}^{D}),\,z \in T^{d} \setminus D). 
\]
\end{lem}
\begin{proof}
Since $T^{d} \cap D$ has Lebesgue measure zero and since $R_{j}(z)$ is bounded on $T^{d} \setminus D$, 
the right-hand side of the above expression defines a bounded operator $R_{j}$ on $L^{2}(T^{d},\mb{C}^{D})$ 
which commutes with $\ucal$. By Lemma $\ref{unitP1}$, the image of $R_{j}$ is contained in that of $\pi_{j}$. 
But, it is obvious that, for $w \in {\rm Im}(\pi_{j})$, $w(z)$ is an eigenvector of $\wh{U}(z)$ with eigenvalue $\omega_{j}$ when $w(z) \neq 0$. 
Thus, we have $R_{j}w=w$, which completes the proof. 
\end{proof}

For each $z_{o} \in T^{d} \setminus D$, we take a small neighborhood $V_{z_{o}}$ such that $\ol{V_{z_{o}}}$ is a compact subset 
contained in a neighborhood of $z_{o}$ as in Lemmas $\ref{eigen1}$, $\ref{eigenP}$. 
We set $W_{z_{o}}=V_{z_{o}} \cap T^{d}$. We may assume that $W_{z_{o}}$ is connected. 
Let $\hcal(W_{z_{o}})$ be the closed subspace in $\hcal$ consisting 
of all $w \in L^{2}(T^{d},\mb{C}^{D})$ whose essential support is contained in $W_{z_{o}}$. 
It is obvious that the orthogonal projection onto $\hcal(W_{z_{o}})$ is given by the multiplication 
of the characteristic function of $W_{z_{o}}$, and hence it commutes with $\ucal$. 
Thus, the spectral projection of $\ucal$ also commutes with the orthogonal projection onto $\hcal(W_{z_{o}})$. 
Let $R_{j}(z)$, $P_{i}(z)$ ($1 \leq j \leq K$, $1 \leq i \leq L$) denote the projection-valued holomorphic function on $V_{z_{o}}$ 
given in Lemma $\ref{eigenP}$. As in \cite{Ka}, we have 
\[
\sum_{j=1}^{K}R_{j}(z) +\sum_{i=1}^{L}P_{i}(z)=I \quad (z \in V_{z_{o}}). 
\]
By Lemma $\ref{unitP2}$, for each $w \in \hcal(W_{z_{o}})$, we see 
\[
w(z)=(\pi_{p}w)(z)+\sum_{i=1}^{L}P_{i}(z)w(z)
\]
for almost all $z \in W_{z_{o}}$. Since $w=\pi_{p}w +\pi_{c}w$, we have 
\begin{equation}\label{pic1}
(\pi_{c}w)(z)=\sum_{i=1}^{L}P_{i}(z)w(z) \quad (w \in \hcal(W_{z_{o}})).
\end{equation}
From $\eqref{pic1}$, we obtain 
\begin{equation}
\ucal^{n}\pi_{c}w(z)=\sum_{i=1}^{L}\lambda_{i}(z)^{n}P_{i}(z)w(z) \quad (w \in \hcal(W_{z_{o}})) 
\end{equation}
for any integer $n$. Since the ring of Laurent polynomials on $S^{1}$ is dense in the space $C(S^{1})$ of 
continuous functions on $S^{1}$ with respect to the supremum norm, we see 
\begin{equation}\label{usp11}
f(\ucal)\pi_{c}w(z)=\sum_{i=1}^{L}f(\lambda_{i}(z)) P_{i}(z)w(z)\quad (w \in \hcal(W_{z_{o}}))
\end{equation}
for any $f \in C(S^{1})$. 
\begin{prop}\label{LOCAL}
Let $V_{z_{o}} \subset T_{\mb{C}^{D}}^{d} \setminus D$ and $W_{z_{o}} \subset T^{d} \setminus D$ be as above. 
Let $w \in \hcal(W_{z_{o}})$. Then the spectral measure $\|E(\cdot)\pi_{c}w\|^{2}$ of $\pi_{c}w$ is absolutely continuous 
with respect to the Lebesgue measure.
\end{prop}
\begin{proof}
We first show that the gradient $\nabla \lambda_{k}$ of the eigenvalues $\lambda_{k}(z)$ ($1 \leq k \leq L$), considered as a 
function on $W_{z_{o}}$, does not vanish almost everywhere on $W_{z_{o}}$. 
To show this, let $(\theta_{1},\ldots,\theta_{d}) \in \mb{R}^{d}$ be a coordinates on $W_{z_{o}}$ 
defined so that $z=(e^{i\theta_{1}},\ldots,e^{i\theta_{d}}) \in W_{z_{o}}$. We denote $\partial_{\theta_{j}}$ the partial derivative 
with respect to $\theta_{j}$. Then $\partial_{\theta_{j}}\lambda_{k}=iz_{j}\partial_{j}\lambda_{k}$ on $W_{z_{o}}$, 
where $\partial_{j}\lambda_{k}$ denotes the derivative in $z_{j}$ of the holomorphic function $\lambda_{k}$. 
Let $N_{k} \subset W_{z_{o}}$ be the set of points $z \in W_{z_{o}}$ where $\nabla \lambda_{k}(z)=0$.  
Thus, $N_{k}$ is the intersection of the sets $(\partial_{j} \lambda_{k})^{-1}(0) \cap T^{d}$ for $j=1,\ldots,d$. 
If $N_{k}$ has positive Lebesgue measure on $T^{d}$, then by Lemma $\ref{analN}$, $\partial_{j}\lambda_{k}$ vanishes identically on $V_{z_{o}}$ 
for any $j=1,\ldots,d$. Therefore, $\lambda_{k}$ is constant on $V_{z_{o}}$. Now, we have $\pi(\lambda_{k},z)=0$ in $z \in V_{z_{o}}$. 
By the identity theorem for holomorphic functions, $\pi(\lambda_{k},z)=0$ for any $z \in T_{\mb{C}}^{d}$. 
This means that the irreducible monic polynomial $\zeta -\lambda_{k} \in k_{d}[\zeta]$ divides $\pi$, which is a contradiction 
to the fact that each $\pi_{\rho}$ in $\eqref{decP2}$ is not of the form $\zeta -\lambda$. 
Therefore, $\nabla \lambda_{k}$ does not vanish almost everywhere on $W_{z_{o}}$. 

Now, let $f$ be a continuous function on $S^{1}$. Since $P_{i}(z)$ and $P_{j}(z)$ ($i \neq j$) are orthogonal to each other for $z \in W_{z_{o}}$, 
$\eqref{usp11}$ shows 
\[
\|f(\ucal)\pi_{c}w\|^{2}=\sum_{i=1}^{L}\int_{W_{z_{o}}} |f(\lambda_{i}(z))|^{2} \|P_{i}(z)w(z)\|_{\mb{C}^{D}}^{2}\,d\nu_{d}(z).
\]
As is shown in the above, we have $\nu_{d}(W_{z_{o}})=\nu_{d}(W_{z_{o}} \cap N_{i}^{c})$ where $N_{i}^{c}$ is the complement of $N_{i}$. 
The function $\lambda_{i}$ restricted to $W_{z_{o}} \cap N_{i}^{c}$ has no critical points, and hence 
for each $t \in S^{1}$, the set $\lambda_{i}^{-1}(t) \cap W_{z_{o}} \cap N_{i}^{c}$ is a smooth hyper-surface. 
For each $t \in S^{1}$ and $i=1,\ldots,K$, we set 
\[
\Gamma_{i}(t)=\int_{\lambda_{i}^{-1}(t) \cap W_{z_{o}} \cap N_{i}^{c}} 
\|P_{i}(z)w(z)\|_{\mb{C}^{D}}^{2}\,\frac{dS_{t}(z)}{|\nabla \lambda_{i}(z)|}
\]
if $\lambda_{i}^{-1}(t) \cap W_{z_{o}} \cap N_{i}^{c} \neq \emptyset$, 
where $dS_{t}$ denotes the volume element of the hyper-surface $\lambda_{i}^{-1}(t) \cap W_{z_{o}} \cap N_{i}^{c}$. 
Then the coarea formula (see \cite{Ch} or \cite[Appendix A]{JL}) shows that $\Gamma_{i}(t)$ is an $L^{1}$-function on $S^{1}$ and  
\[
\int_{S^{1}}|f(t)|^{2}\,d\|E(t)\pi_{c}w\|^{2}=\|f(\ucal)\pi_{c}w\|^{2}=\int_{S^{1}} |f(t)|^{2}\Gamma(t)\,d\nu_{1}(t),\quad 
\Gamma(t)=\sum_{i=1}^{L}\Gamma_{i}(t),  
\]
which proves that the measure $d\|E(t)\pi_{c}w\|^{2}$ is absolutely continuous with respect to $d\nu_{1}(t)$. 
\end{proof}

\vspace{10pt}

\noindent{\it Proof of Theorem $\ref{main0}$, {\rm (2)}.}\hspace{2pt}
For each $z_{o} \in T^{d} \setminus D$, we take a small neighborhood $V_{z_{o}}$ such that 
$\ol{V_{z_{o}}}$ is a compact set contained in a neighborhood of $z_{o}$ as in Lemmas $\ref{eigen1}$, $\ref{eigenP}$. 
We set, as before, $W_{z_{o}}=T^{d} \cap V_{z_{o}}$. 
Then $\{W_{z_{o}}\}_{z_{o} \in T^{d} \setminus D}$ is an open covering of $T^{d} \setminus D$. 
We can choose a locally finite countable open covering $\{W_{i}\}_{i=1}^{\infty}$ which is a refinement of a countable 
sub-covering of the covering $\{W_{z_{o}}\}_{z_{o} \in T^{d} \setminus D}$. 
We take a partition of unity $\{\rho_{i}\}_{i=1}^{\infty}$ subordinated to the covering $\{W_{i}\}_{i=1}^{\infty}$. 
Then, the sum $\dsp \sum_{i=1}^{\infty}M_{i}$ converges in the strong operator topology and equals the identity operator on $\hcal$. 
Let $w \in \hcal$. Let $\Lambda \subset S^{1}$ be a Borel subset having Lebesgue measure zero. 
Since $\ucal$ commutes with the operator $M_{i}$ defined by the multiplication by $\rho_{i}$, 
the spectral measure $E(\Lambda)$ and the projection $\pi_{c}$ also commute with $M_{i}$. 
Therefore, we obtain 
\[
\ispa{E(\Lambda)\pi_{c}w,\pi_{c}w}=\sum_{i=1}^{\infty}\int_{T^{d} \setminus D}\|E(\Lambda)\pi_{c}(\rho_{i}^{1/2}w)(z)\|_{\mb{C}^{D}}^{2}\,d\nu_{d}(z), 
\]
Now Proposition $\ref{LOCAL}$ shows that $E(\Lambda)\pi_{c}(\rho_{i}^{1/2}w)=0$ in $\hcal$. Therefore, we have $E(\Lambda)\pi_{c}w=0$, 
which completes the proof of Theorem $\ref{main0}$, (2). 
\hfill$\square$

\end{document}